\documentclass[a4paper,12pt]{article}

\usepackage[top=2.5cm,bottom=2.5cm,left=2.5cm,right=2.5cm]{geometry}%{geometry} 表示可以自定义页面设置
\usepackage{amsmath,amsthm,mathrsfs,graphicx,amsfonts}%{mathrsfs}用于产生一种数学用的花体字，数学公式宏包
\usepackage{bm}%处理数学公式中的黑斜体，公式中的粗体字符（用命令\boldsymbol）
\usepackage{cases}%\begin{numcases}{|x|=}x,&for$x\geq0$\\-x,&for$x<0$\end{numcases}
\usepackage[english]{babel}%算法
\usepackage{amsfonts}
\usepackage{latexsym}
\usepackage{graphicx}%图形宏包
\usepackage{txfonts}
\usepackage[numbers,sort&compress]{natbib}
\usepackage[natural]{xcolor}
\usepackage{rotating}
\usepackage{mathtools}

\allowdisplaybreaks
\numberwithin{equation}{section}

\newtheorem{theorem}{Theorem}[section]

\newtheorem{corollary}[theorem]{Corollary}
\theoremstyle{definition}

\theoremstyle{remark}

\newcommand{\fo}{f_{\mathrm{o}}}
\newcommand{\dd}{\mathbin{\dot\cup}}
\newcommand{\E}{\mathbb{E}}

\begin{document}
\title{Large odd induced subgraphs via odd cuts}

% Anonymous submission version. Replace the author block before a
% non-anonymous submission and add \address and \email as appropriate.
\author{Qinghou Zeng\footnote{Center for Discrete Mathematics, Fuzhou University, Fuzhou, Fujian 350003, China. Research supported by National Key R\&D Program of China (Grant No.~2023YFA1010202) and National Natural Science Foundation of China (Grant No.~12371342). Email: zengqh@fzu.edu.cn.}}
\date{}
\maketitle
\begin{abstract}
%In this paper, we show that every $n$-vertex graph without isolated vertices has an induced subgraph with all degrees odd on at least $n/5$ vertices. The key ingredient is analogous to Gallai's even-even partition theorem, and we provide an odd-odd partition result for graphs possessing odd cuts.

%A classical result of Gallai asserts that every graph has a vertex partition into two sets, each inducing a subgraph with all degrees even. We prove that if a graph has a bipartition such that every vertex has an odd number of neighbors on the opposite side, then it can be partitioned into two sets, each inducing a subgraph with all degrees odd. As an application, we show that every $n$-vertex graph without isolated vertices has an induced subgraph with all degrees odd on at least $n/5$ vertices.

%Gallai proved that every graph can be partitioned into two sets, each inducing a subgraph with all degrees even. We show that if a graph admits a bipartition in which every vertex has odd cross-degree, then it can be partitioned into two sets, each inducing a subgraph with all degrees odd. Consequently, every (n)-vertex graph without isolated vertices has an induced subgraph with all degrees odd on at least (n/5) vertices.
Gallai proved that every graph can be partitioned into two sets, each
inducing a subgraph with all degrees even. We show that if a graph
admits a bipartition in which every vertex has an odd number of
neighbors in the opposite part, then it can be partitioned into two sets,
each inducing a subgraph with all degrees odd. Consequently, every
$n$-vertex graph without isolated vertices has an induced subgraph with
all degrees odd on at least $n/5$ vertices, substantially improving the previously known
universal lower bound.
\end{abstract}
%\footnotetext{2010 Mathematics Subject Classification: primary 05C35; secondary 05C70.}
%{odd induced subgraph, domination number, private neighbor, degree parity, probabilistic method}



\section{Introduction}

All graphs in this paper are finite and simple. A graph is called
\emph{even} or \emph{odd} if every vertex has even or odd degree,
respectively.
We begin with a classical partition theorem of Gallai; see, for
example, Lov\'asz~\cite[Problem~5.17]{Lov2007}.

\begin{theorem}[Gallai's Theorem]\label{thm:gallai}
Every graph $G$ admits a partition $V(G)=V_1\dd V_2$ such that both
$G[V_1]$ and $G[V_2]$ are even.
\end{theorem}

Note that there is no unconditional odd analogue of this theorem
since any graph with an odd number of vertices cannot be
partitioned into two odd induced subgraphs. If an arbitrary
number of parts is allowed, Scott~\cite{Sco2001} proved that a graph
can be partitioned into odd induced subgraphs if and only if each of
its components has even order.
Such partitions have since been studied
from both structural and algorithmic viewpoints; see
\cite{AAG2023,BS2021,
BHK2025}.
%Such partitions, and the associated odd
%chromatic number, have since been studied further; see, for example, Belmonte, Harutyunyan, K{\"o}hler and Melissinos~\cite{BHK2025}.
We show that a natural cut condition
guarantees that two parts suffice.
%Thus, every odd-cut set induces a graph of odd chromatic number at most two.
For a graph $G$, we call a set $S\subseteq V(G)$  an \emph{odd-cut set} if it admits a
partition $S=P\dd Q$ such that every vertex of $S$ has an odd number
of neighbors in the opposite part.

\begin{theorem}\label{thm:lifting}
Let $G$ be a graph and let $S$ be an odd-cut set of $G$. Then there is
a partition $S=S_0\dd S_1$ such that both $G[S_0]$ and $G[S_1]$ are
odd.
\end{theorem}

Gallai's partition immediately yields an even induced subgraph on at
least half of the vertices, and paths show that the factor $1/2$ is
best possible. This motivates the corresponding extremal question for
odd induced subgraphs. Note that an isolated vertex cannot belong to an odd induced subgraph.
Thus, we restrict attention to graphs without isolated vertices.

For a graph $G$, let
\[
   \fo(G):=\max\bigl\{|S|:S\subseteq V(G)
   \text{ and }G[S]\text{ is odd}\bigr\}.
\]
The problem of bounding $\fo(G)$ was studied by
Caro~\cite{Car1994} and Scott~\cite{Sco1992}, who established lower
bounds of order $\sqrt n$ and $n/\log n$, respectively. Resolving the
long-standing conjecture (cited by Caro \cite{Car1994} as ``part of the graph theory
folklore") that a linear bound holds, Ferber and
Krivelevich~\cite{FK2022} proved the bound $n/10000$.
Very recently, Gutin, Hao, and Zhou~\cite{GHZ2026}
proved a more general result implying the bound $2n/21$. Sharper bounds are known for several graph classes, including
trees, graphs of bounded maximum degree or treewidth, planar graphs of
large girth, line graphs, and claw-free
graphs~\cite{RS1995,BWW1997,HYL2018,
RHZ2022,WW2024,Nin2026}.

By Theorem~\ref{thm:lifting}, every odd-cut set $S$ yields an odd
induced subgraph on at least $|S|/2$ vertices. Thus, the problem reduces
to constructing large odd-cut sets. Together with Scott's bipartite
bound~\cite{Sco1992}, this observation gives a short proof of the
following consequence.

\begin{corollary}\label{cor:eighth}
Every $n$-vertex graph $G$ without isolated vertices satisfies
\[
   \fo(G)\ge \frac n8.
\]
\end{corollary}

This argument uses only a spanning bipartite subgraph of $G$. By
constructing larger odd-cut sets through domination, we obtain a
stronger bound stated in our main theorem.

\begin{theorem}\label{thm:main}
Every $n$-vertex graph $G$ without isolated vertices satisfies
\[
   \fo(G)\ge \frac n5.
\]
\end{theorem}

\noindent\textbf{Notation.}
Let $G=(V(G),E(G))$ be a graph. For any $v\in V(G)$, let $N_G(v)$ and $d_G(v)$ denote its \emph{neighborhood} and
\emph{degree}, respectively. For $X\subseteq V(G)$, let $G[X]$ denote the
subgraph induced by $X$, and set
$d_G(v,X):=|N_G(v)\cap X|$. For disjoint $X,Y\subseteq V(G)$, let
$G[X,Y]$ denote the bipartite subgraph of $G$ consisting of all edges
between $X$ and $Y$.
%\newpage

\section{Odd cuts and a universal lower bound}\label{sec:odd-cuts}
In this section, we prove Theorem~\ref{thm:lifting} and Corollary~\ref{cor:eighth}.

\begin{proof}[{\bf Proof of Theorem \ref{thm:lifting}}]
Choose a partition $S=P\dd Q$ such that every vertex of $S$ has an
odd number of neighbors in the opposite part. Applying
Theorem~\ref{thm:gallai} to $G[S]$, there is a partition
$S=U\dd W$ such that both $G[U]$ and $G[W]$ are even. Let
\[
   S_0:=(U\cap P)\cup(W\cap Q)
   \;\,\text{and}\;\,
   S_1:=(W\cap P)\cup(U\cap Q).
\]
Clearly, $S=S_0\dd S_1$. Now, we show that both $G[S_0]$ and $G[S_1]$ are odd.

Fix $i\in\{0,1\}$ and $v\in S_i$. Let $A\in\{U,W\}$ be the part containing $v$, and let $B\in\{P,Q\}$ be the part not containing
$v$. Clearly, we have $S_i=A\triangle B$ by the construction of $S_i$, where $\triangle$ denotes the symmetric difference. Hence
\[
\begin{aligned}
   d_{G[S_i]}(v)
   =d_G(v,A\triangle B)
   \equiv d_G(v,A)+d_G(v,B)
   \equiv 0+1
   \equiv 1\pmod 2.
\end{aligned}
\]
Indeed, $d_G(v,A)$ is even because $G[A]$ is even, while $d_G(v,B)$
is odd by the choice of $P\dd Q$. Hence, both $G[S_0]$ and $G[S_1]$
are odd, completing the proof of Theorem \ref{thm:lifting}.
\end{proof}

\begin{proof}[{\bf Proof of Corollary \ref{cor:eighth}}]
Choose a spanning tree in each component of $G$, and combine their
bipartitions into a partition $V(G)=P\dd Q$. Let $B=G[P,Q]$ be the spanning bipartite graph consisting of all edges crossing the cut.
Every vertex is incident with an edge of one of the spanning trees, so
$B$ has no isolated vertices. Note that Scott \cite{Sco1992} proved that $\fo(B)\ge |V(B)|/4$ for every bipartite graph $B$ without isolated vertices. It follows that there is a
set $S\subseteq V(G)$ such that $B[S]$ is odd and $|S|\ge n/4$.
Since $B[S]=G[S\cap P,S\cap Q]$,
every vertex of $S$ has an odd number of neighbors in the opposite part
of the partition  $S=(S\cap P)\dd(S\cap Q)$.
Thus, $S$ is an odd-cut set of $G$, and $f_o(G)\ge |S|/2\ge n/8$ by Theorem \ref{thm:lifting}.
\end{proof}

\section{Large odd cuts from domination}\label{sec:domination}
A set $D\subseteq V(G)$ is \emph{dominating} if every vertex outside
$D$ has a neighbor in $D$. The minimum cardinality of a dominating set
is the \emph{domination number} $\gamma(G)$. We use the following
theorem of Bollob\'as and Cockayne~\cite{BC1979}; see also
\cite[Lemma~4.22, p.~80]{HHH2023} for a proof.

\begin{theorem}[Bollob\'as and Cockayne~\cite{BC1979}]\label{thm:private-domination}
Every graph without isolated vertices has a minimum dominating set $D$
such that, for every $u\in D$, there exists a vertex
$v\in V(G)\setminus D$ satisfying $N_{G}(v)\cap D=\{u\}$.
\end{theorem}

Let $G$ be an $n$-vertex graph without isolated vertices, and fix a
minimum dominating set $D$ as in
Theorem~\ref{thm:private-domination} so that $|D|=\gamma(G)$. For each $u\in D$, let $Q_u:=\{v\in V(G)\setminus D:N_{G}(v)\cap D=\{u\}\}$.
By Theorem~\ref{thm:private-domination}, the sets $Q_u$ with $u\in D$ are nonempty and pairwise disjoint. Let $Q:=\cup_{u\in D}Q_u$ and $R:=V(G)\setminus(D\cup Q)$.
Then
\begin{equation}\label{eq:decomposition}
   n=|D|+|Q|+|R|.
\end{equation}
%and every vertex of $R$ has at least two neighbors in $D$.

\begin{theorem}\label{thm:parameterized}
If $G$ is an $n$-vertex graph without isolated vertices, then
\[
   \fo(G)\ge
   \max\left\{\gamma(G), \frac{n-\gamma(G)}{4}\right\}.
\]
\end{theorem}
\begin{proof}[{\bf Proof}]
We first show that $\fo(G)\ge\gamma(G)$. For each $u\in D$, choose a vertex $q_u\in Q_u$. By the definition of
$Q_u$, the partition $D\dd\{q_u:u\in D\}$ is an odd cut: every vertex
has exactly one neighbor in the opposite part. Hence,
Theorem \ref{thm:lifting} gives $\fo(G)\ge |D|=\gamma(G)$.

Now, we prove that $\fo(G)\ge(n-\gamma(G))/4$. Choose a random set $Z\subseteq D$ by
including each vertex independently with probability $1/2$, and let $R_Z:=\{v\in R:d_{G}(v,Z)\text{ is odd}\}$.
For each $u\in Z$, choose a largest set $T_u\subseteq Q_u$ satisfying
\begin{equation}\label{eq:parity-choice}
   |T_u|+d_{G}(u,R_Z)\equiv 1\pmod 2.
\end{equation}
Since $Q_u\ne\varnothing$, the required parity can be obtained by
deleting at most one vertex, and hence
\begin{equation}\label{eq:private-size}
   |T_u|\ge |Q_u|-1.
\end{equation}

Set $S_Z:=Z\cup R_Z\cup\bigcup_{u\in Z}T_u.$ We claim that $S_Z=Z\dd(S_Z\setminus Z)$ is an odd cut. Every vertex
of $R_Z$ has an odd number of neighbors in $Z$, while every vertex of
$T_u$ has exactly one neighbor in $Z$, namely $u$. Moreover, no vertex of $T_w$ with $w\ne u$ is adjacent to $u$. Hence, for each
$u\in Z$, $d_{G}(u,S_Z\setminus Z)=d_{G}(u,R_Z)+|T_u|\equiv 1\pmod 2$
by~\eqref{eq:parity-choice}. Thus $S_Z$ is an odd-cut set.

Using \eqref{eq:private-size}, we obtain
\begin{equation}\label{eq:SZ-lower}
\begin{aligned}
   |S_Z|
   =|Z|+|R_Z|+\sum_{u\in Z}|T_u|
   &\ge |R_Z|+\sum_{u\in Z}|Q_u|.
\end{aligned}
\end{equation}
For each $u\in D$, we have $\mathbb P(u\in Z)=1/2$. Also, for every
$v\in R$, we have $\mathbb P(v\in R_Z)=1/2$. Indeed, fix a neighbor
$x\in N_{G}(v)\cap D$ and expose all choices except whether $x$ belongs to
$Z$; exactly one of the two possibilities makes $d_{G}(v,Z)$ odd.
Therefore, taking expectations in~\eqref{eq:SZ-lower} and using
linearity of expectation, we obtain
\[
\begin{aligned}
   \E|S_Z|
   &\ge
   \sum_{v\in R}\mathbb P(v\in R_Z)
   +\sum_{u\in D}|Q_u|\,\mathbb P(u\in Z)\\
   &=\frac{1}{2}|R|
     +\frac12\sum_{u\in D}|Q_u|=\frac{|R|+|Q|}{2}
    =\frac{n-|D|}{2},
\end{aligned}
\]
where the last equality follows from \eqref{eq:decomposition}. Hence
some choice of $Z$ gives an odd-cut set $S_Z$ with $|S_Z|\ge(n-|D|)/2$.
Again, Theorem \ref{thm:lifting} then yields
\[
   \fo(G)\ge\frac{|S_Z|}{2}
   \ge\frac{n-|D|}{4}
   =\frac{n-\gamma(G)}{4}.
\]
Hence, we complete the proof of Theorem \ref{thm:parameterized}.
\end{proof}

\begin{proof}[{\bf Proof of Theorem~\ref{thm:main}}]
Since the maximum of two numbers is at least any convex combination
of them, Theorem~\ref{thm:parameterized} gives
\[
\fo(G)\ge\frac15\cdot\gamma(G)+\frac45\cdot\frac{n-\gamma(G)}{4}=\frac n5,
\]
completing the proof of Theorem \ref{thm:main}.
\end{proof}

\section{Concluding remarks}\label{sec:conclusion}
The choice of probability $1/2$ in the proof of
Theorem~\ref{thm:parameterized} was chosen for simplicity. The same
argument can be slightly sharpened by using an appropriately chosen
inclusion probability $p\in[1/2,1]$. Since the resulting improvement
over $n/5$ is marginal, we retain the cleaner bound in
Theorem~\ref{thm:main} to keep the proof particularly transparent.

As observed by Caro~\cite{Car1994}, the complement of the $7$-cycle shows that the optimal universal constant
is at most $2/7$. This bound is known to be
sharp for graphs of maximum degree at most four; see Ai, Guo, Gutin, Hao and Yeo~\cite{AGG2025}.
 Together with Theorem~\ref{thm:main}, this places
the optimal constant for general graphs without isolated vertices between $1/5$ and
$2/7$. Determining its exact value remains open. More broadly, the
odd-cut viewpoint may be useful in parity problems where cross-degrees
are easier to control than degrees within induced subgraphs.

\section*{Declaration on the Use of Generative AI}
The author used ChatGPT (OpenAI) for language polishing, structural
organization, and improving the exposition of some proofs. All mathematical arguments and proofs
were independently checked and finalized by the author, who takes full
responsibility for the manuscript.

\end{document}